\documentclass[10pt,reqno]{amsart} 

\usepackage{amssymb,latexsym}
\usepackage{amsmath}
\usepackage{amsthm}
\usepackage{cite}
\usepackage{graphicx}
\usepackage[height=190mm,width=130mm]{geometry}

\newtheorem{corollary}{Corollary}
\newtheorem{proposition}{Proposition}
\newtheorem{remark}{Remark}

\newcommand{\half}{\frac{1}{2}}

\title[Determinant Representations for Garvan Formula]
{Determinant Representations for Garvan Formula}

\author{D. Levin}
\address{Mathematical Institute, University of Oxford, Andrew Wiles Building, 
Radcliffe Observatory Quarter (550), Woodstock Road, Oxford, OX2 6GG, United Kingdom 
(Currently working in Israel)}
\email{levindanie@gmail.com}

\author{H.-G. Shin}
\address{Faculty of Mathematics and Physics, Charles University, Sokolovsk{\'a} 83, 
Prague, Czech Republic}
\email{shgmath@gmail.com}

\author{A. Zuevsky}
\address{Institute of Mathematics, Czech Academy of Sciences, Prague, Czech Republic}
\email{zuevsky@yahoo.com}


\begin{document}

\begin{abstract}
In this note, we demonstrate how determinant representations for correlation functions 
in conformal field theory can be used to derive explicit determinant formulas for powers 
of the classical $\eta$-function, expressed via deformed elliptic functions with parameters. 
At genus one, comparing the fermionic and bosonic evaluations of a twisted torus partition 
function, together with an elliptic Frobenius/Fay determinant identity, gives for every 
$n\ge1$ a determinant formula for $\eta(\tau)^{3n}$ built from a twisted Szeg\H{o} kernel; 
at $n\equiv0,1\pmod 8$ this gives determinant identities for powers of the modular 
discriminant $\Delta=\eta^{24}$ of the same shape as the classical formulas of 
Garvan and Milne, although those classical determinants are not recovered. 
At genus two, an analogous comparison for a torus formed by self-sewing, combined with a 
finite/infinite determinant factorization of the genus two Szeg\H{o} kernel, 
gives a determinant formula for $\eta(\tau)^{3\kappa^2}$, $-1/2<\kappa<1/2$, 
in terms of genus two data; identifying the genus two modular form to which this specializes 
is left open.
\end{abstract}

\maketitle

%%%%%%%%%%%%%%%%%%%%%%%%%%%%%%%%%%%%%%%%%%%%%%%%%%%%%%%%%%%%%%%%%%%%%%%%%%%%%%%%%%%
\section{Introduction}

Many identities in number theory originate from computations of correlation functions 
in conformal field theory \cite{DFG}, expressed through the language of vertex algebras
 \cite{Kac, Zhu}. In particular, one finds relations among modular forms, fundamental 
kernels, and $q$-series from the comparison of bosonic and fermionic pictures. The most 
interesting example in this direction is represented by the twisted counterpart of the 
Jacobi triple identity \cite{Kac}.

In computations of vertex operator algebra (VOA) correlation functions on the torus, 
classical Weierstrass functions and Eisenstein series occur. Their two-parameter 
generalizations were introduced in \cite{DLM} and further developed in \cite{Mason-Tuite-Z}. 
For higher correlation functions on a genus one Riemann surface, computations involve 
elliptic versions of Fay trisecant identity \cite{Fay}, known from algebraic geometry. 
Various identities for powers of the $\eta(\tau)$-function appear in studies of affine 
Lie algebras \cite{Kac, KacInfinite, Macdonald} and of Ramanujan-type Eisenstein series 
identities \cite{ChanCooperToh}.

In the genus two case, the bosonization procedure 
\cite{MasonTuiteGenus2-I, MasonTuiteGenus2-II} provides us with genus two counterparts 
of Jacobi triple identities in terms of determinants of reproduction kernels corresponding 
to the type of vertex operator algebra used.

In this note, using the relations mentioned above, we derive explicit determinant formulas
 generalizing the Garvan identity for elliptic Eisenstein series, first to arbitrary powers 
of $\eta(\tau)$ at genus one and then, via torus self-sewing, to a genus two determinant 
identity for $\eta(\tau)^{3\kappa^2}$.

\subsection{Torus correlation functions}
For a $g$-twisted module $\mathcal V$ for a vertex operator superalgebra $V$, we find closed 
formulas for correlation functions of vertex operators ${\mathcal Y}$ on a cylinder, 
$q = e^{2 \pi i \tau}$, with local coordinates $z_i$, $v_i \in V$, $1 \le i \le n$:
\[
Z^{(1)}_V\left[
\begin{array}{c}
f \\
g
\end{array}
\right](v_1, z_1, \ldots, v_n, z_n; q) =
{\rm STr}_{\mathcal V} \left (f{\mathcal Y} (v_1, z_1) \ldots {\mathcal{Y}} (v_n, z_n) \;
q^{L(0)-C/24} \right),
\]
where $L(0)$ is the Virasoro algebra generator, and $C$ is the central charge. The formal 
parameter is associated with a complex parameter on the torus. Final expressions are given 
by determinants of matrices with elements being coefficients in the expansions of the regular 
parts of corresponding differentials: Bergman (bosons) or Szeg\H{o} (fermions) kernels 
\cite{Mason-Tuite-Z}.

In this note, we derive explicit genus two generalizations of the fundamental formulas 
for powers of the $\eta$-function in terms of deformed versions \cite{DLM} of Weierstrass 
functions and Eisenstein series. In particular, we find that powers of the modular 
discriminant are expressed (up to theta-function multipliers) via determinants of finite 
matrices containing combinations of deformed modular functions. In the proof, we use the 
generalized elliptic version of Fay trisecant identity for a vertex operator superalgebra.

%%%%%%%%%%%%%%%%%%%%%%%%%%%%%%%%%%%%%%%%%%%%%%%%%%%%%%%%%%%%%%%%%%%%%%%%%%%%%%%%%%
\section{Modular Discriminant and Eisenstein Series}

The modular discriminant is defined by $\Delta(\tau)= \eta(\tau)^{24}$, where $\eta(\tau)$ 
is the Dedekind eta-function, $q = e^{2 \pi i \tau}$, and 
$\eta (\tau )=q^{1/24}\prod_{n=1}^\infty (1-q^n)$.

Recall the Eisenstein series \cite{Ramanujan, Serre}; we set $E^{(1)}_n=0$ for $n$ odd, 
and for $n$ even:
\begin{equation}
\label{eq:Eisenstein_Def}
E^{(1)}_n(\tau)=-\frac{B_n(0)}{n!}+\frac{2}{(n-1)!}
\sum\limits_{r\geq 1}\frac{r^{n-1}q^r}{1-q^r},
\end{equation}
where $B_n(0)$ is the $n$-th Bernoulli number, defined by:
\begin{equation*}
\frac{q_z^{\lambda }}{q_z-1}=\frac{1}{z}+\sum\limits_{n\geq 1}
\frac{B_n(\lambda )}{n!}z^{n-1}.
\end{equation*}
Note that the superscript $(1)$ in the definition of $E^{(1)}_n(\tau)$ denotes the genus one 
(elliptic) case of higher genus generalizations of the Eisenstein series.

In the classical normalization $E_n(q=0)=1$ recalled in Section \ref{sec:classical-garvan} 
below (not the $E^{(1)}_n$ of \eqref{eq:Eisenstein_Def}, for which 
$E^{(1)}_n(q=0)=-B_n(0)/n!$), one finds \cite{Ramanujan} the relations
$$E_8 = (E_4(\tau))^2,\qquad E_{10}(\tau) = E_4(\tau) E_6(\tau),\qquad E_{12}(\tau) 
= \tfrac{441}{691}(E_4(\tau))^3 + \tfrac{250}{691}(E_6(\tau))^2.$$
(In the $(1)$-normalization these become $E^{(1)}_8=\tfrac37(E^{(1)}_4)^2$, $E^{(1)}_{10}
=\tfrac{5}{11}E^{(1)}_4E^{(1)}_6$, $E^{(1)}_{12}=\tfrac{18}{143}(E^{(1)}_4)^3
+\tfrac{25}{143}(E^{(1)}_6)^2$; we do not use this form below.)

\subsection{Classical Garvan formula}\label{sec:classical-garvan}
The fundamental classical formulas for the modular discriminant follow 
(the Hankel-determinant form of the second is due to Garvan, see \cite[(1.5), (1.6)]{Milne}):
\[
\eta^{24}(\tau)=\frac{1}{1728} \left( \left(E_4(\tau)\right)^3
-\left(E_6(\tau)\right)^2 \right)
=\frac{1}{1728} \det
\left[
\begin{array}{cc}
E_4(\tau) & E_6(\tau) \\
E_6(\tau) & E_8(\tau)
\end{array}
\right].
\]
Note the difference in normalization. The Eisenstein elements $E^{(1)}_n$ from definition 
\eqref{eq:Eisenstein_Def} satisfy $E^{(1)}_n(q = 0) = - \frac{B_n(0)}{n!}$, whereas the given 
Ramanujan identities, the formula for $\Delta (\tau) = \eta(\tau)^{24} = q - 24q^2 + \ldots$, 
and the Garvan formula all employ the more commonly used normalization $E_n(q = 0) = 1$. 
For the more common normalization, we use $E_n$ without a superscript.

Indeed, the fact that the Garvan $3 \times 3$ determinant is a cusp form with leading $O(q^2)$ 
expansion follows from the $E_n(q = 0) = 1$ normalization. Similar remarks apply 
to the appearance of cusp forms in all generalized Garvan identities of Milne \cite{Milne}.

The next formula was conjectured by F. Garvan (private communication 
to S. C. Milne, March 1997) and proved by S. C. Milne: it is Theorem 1.4 
(``Garvan identity'', equation (1.7)) of \cite{Milne}; see also the fuller account 
of Hankel determinants of Eisenstein series in \cite{MilneMemoir}:
\begin{equation}
\label{eq:Garvan_3x3}
\eta(\tau)^{48}=-\frac{691}{250 \; (1728)^2 } \;
\det
\left[
\begin{array}{ccc}
E_4(\tau) & E_6(\tau) & E_8 (\tau) \\
E_6(\tau) & E_8(\tau) & E_{10}(\tau) \\
E_8(\tau) & E_{10}(\tau) & E_{12}(\tau)
\end{array}
\right].
\end{equation}
Milne general result \cite[Theorem\; 2.5]{Milne} (see also \cite{MilneMemoir}) evaluates 
the $n\times n$ Hankel determinant $\det(E_{2(i+j)}(\tau))_{1\le i,j\le n}$, of weight 
$2n(n+1)$, as $\Delta(\tau)^{n-1}$ times a homogeneous polynomial of weight $2(n-2)(n-3)$ 
in $E_4(\tau)^3$ and $E_6(\tau)^2$ (and $E_4$ when needed); \eqref{eq:Garvan_3x3} is the 
case $n=3$. This polynomial factor is forced by the weight count just given - it is the 
content of Milne theorem, not an artefact of a suboptimal choice of basis for 
$M_*(\mathrm{SL}_2(\mathbb Z))$.

In contrast, our formulas for powers of the $\eta$-function, given below, are expressed 
in terms of deformed Eisenstein series \cite{DLM} rather than as a determinant of classical 
Eisenstein series. Two qualifications are in order. First, the entries 
$P_1^{(1)}\left[ \begin{smallmatrix} \theta \\ \phi \end{smallmatrix} \right](x_i-y_j,\tau)$ 
of our determinants are not modular forms - they depend on $16n$ free complex parameters 
together with $(\alpha,\beta)$ - and \eqref{eq:Garvan_3x3} is not literally recovered as 
a specialization of what follows: we obtain a structurally different, 
vertex-operator-algebraic route to identities of the same shape, in which the power of 
$\eta$ arises from a prime-form count instead of from the ring structure of 
$M_*(\mathrm{SL}_2(\mathbb Z))$. So, although our formulas use a different set of Eisenstein 
series elements, the overall idea of the identities is quite close to the Garvan-Milne 
formulas. Second, if a bridge between the two circles of ideas exists, it more likely runs
 through the Hankel matrices $C^{(1)}(k,l,\tau)$ and $A^{(1)}_a(k,l,\tau_a,\epsilon)$ of 
 Section \ref{sec:genus2} below than through the matrices $\mathbf P_n$, $\mathbf D_{r,s}$ 
used here: since $E^{(1)}_n=0$ for $n$ odd, the surviving entries of $C^{(1)}$ and 
$A^{(1)}_a$ are indexed by $k+l$ even and coincide, up to the normalization of 
\eqref{eq:Eisenstein_Def}, with Milne $E_{2(i+j)}$, so that $\det(I-A_1A_2)$, 
already used in Section \ref{sec:genus2}, is an infinite determinant built from Milne  
objects; whether a finite Hankel minor of it gives a direct vertex-operator-algebra proof 
of Milne theorem we leave open.

In this note, we give various generalizations for higher powers of the $\eta$-function, 
computed as a determinant of matrices containing deformed Weierstrass functions \cite{DLM}.

%%%%%%%%%%%%%%%%%%%%%%%%%%%%%%%%%%%%%%%%%%%%%%%%%%%%%%%%%%%%%%%%%%%%%%%%%%%%%%%%%%
\section{The Generalized Garvan Formulas}\label{sec:garvan}

Computations of the twisted partition function 
$Z_V^{(1)}\left[ \begin{smallmatrix} f \\ g \end{smallmatrix} \right](\tau)$ for the free 
fermion even rank vertex operator superalgebra, where $f, g$ are inner automorphisms 
generated by a Heisenberg element of $V$, lead to two alternative expressions (see, e.g., 
\cite{Kac}) as an expansion over a basis. First:
\begin{equation*}
Z_V^{(1)}\left[
\begin{array}{c}
f \\
g
\end{array}
\right] (\tau )=
q^{\kappa^2/2-1/24}\prod_{l\geq 1}(1-\theta^{-1}
q^{l-\frac{1}{2}-\kappa })(1-\theta q^{l-\frac{1}{2}+\kappa}),
\end{equation*}
and second, in terms of torus theta series with characteristics:
\begin{equation}
\label{eq:Zform2}
Z_V^{(1)}\left[
\begin{array}{c}
f \\
g
\end{array}
\right] (\tau )
=
\frac{e^{2\pi i(\alpha +1/2)(\beta -1/2)}}{\eta (\tau )}\vartheta^{(1)} \left[
\begin{array}{c}
\frac{1}{2} - \beta \\
\frac{1}{2}+ \alpha
\end{array}
\right] (0,\tau ),
\end{equation}
where the theta series is defined as:
\begin{equation*}
\vartheta^{(1)} \left[
\begin{array}{c}
a \\
b
\end{array}
\right] (z,\tau )=\sum\limits_{n\in \mathbb{Z}}
\exp \left(i\pi (n+a)^{2}\tau+(n+a)(z+2\pi ib)\right).
\end{equation*}
Here we define $f=e^{2\pi i\alpha a(0)}$ and $g=e^{2\pi i\beta a(0)}$ with parameters 
$\alpha, \beta \in \mathbb{R}$, where $a(0)$ is the zero mode of a Heisenberg subalgebra 
in the rank two free fermionic vertex operator superalgebra \cite{Mason-Tuite-Z}. 
We also define $\phi =e^{-2\pi i\beta }$ and $\theta =e^{-2\pi i\alpha}$.
Note that $Z_V^{(1)}\left[ \begin{smallmatrix} f \\ g \end{smallmatrix} \right] (\tau )=0$ 
for $(\theta, \phi )=(1,1)$, i.e., $(\alpha, \beta ) \equiv (0,0) \pmod{\mathbb{Z}}$.

The parameter $\kappa$ is fixed by comparing the two expressions just given. By the Jacobi 
triple product, 
$\sum_{n}q^{n^2/2}z^n = \prod_{l\geq1}(1-q^l)(1+zq^{l-1/2})(1+z^{-1}q^{l-1/2})$; 
writing $a=\frac12-\beta$, $b=\frac12+\alpha$ and $z=q^ae^{2\pi ib}$ gives
\begin{equation*}
\frac{\vartheta^{(1)}\left[\begin{smallmatrix}a\\b\end{smallmatrix}\right]
(0,\tau)}{\eta(\tau)} = e^{2\pi iab}q^{a^2/2-1/24}\prod_{l\geq1}
\left(1+e^{2\pi ib}q^{l-\frac12+a}\right)\left(1+e^{-2\pi ib}q^{l-\frac12-a}\right),
\end{equation*}
and since $e^{2\pi ib}=-e^{2\pi i\alpha}=-\theta^{-1}$ and $e^{-2\pi ib}=-\theta$, 
this coincides with the fermionic product above when
\begin{equation}
\label{eq:kappa}
\kappa=\beta-\frac12.
\end{equation}
(At $(\alpha,\beta)\equiv(0,0)$ this gives $\kappa=-1/2$, where the $l=1$ factor 
$1-q^0$ vanishes, consistent with $Z_V^{(1)}=0$ there; and the restriction 
$-1/2<\kappa<1/2$ of Section \ref{sec:genus2} below is $0<\beta<1$.)

With \eqref{eq:kappa}, comparing the two representations gives the Jacobi triple product 
formula \cite{Kac}, which can be rewritten in the form:
\begin{equation*}
\eta (\tau )
=
\frac{
q^{-\kappa^2/2+1/24}  e^{2\pi i(\alpha +1/2)(\beta -1/2)} \;
\vartheta^{(1)} \left[
\begin{array}{c}
\frac{1}{2} - \beta \\
\frac{1}{2} + \alpha
\end{array}
\right] (0,\tau)} {\det \left(I-T^{(1)}\right)},
\end{equation*}
where the determinant corresponding to sphere self-sewing to form a torus \cite{Tuite-Z} is:
\[
\det \left(I-T^{(1)}\right)
= \prod_{l\geq 1} \left(1-\theta ^{-1}q^{l-\frac{1}{2}-\kappa }\right)
\left(1-\theta q^{l- \frac{1}{2}+\kappa }\right).
\]
Thus we get the identity for the first power of the $\eta$-function.

In \cite{DLM}, the deformed Weierstrass functions (which can be expressed via deformed 
Eisenstein series) were defined and studied. Writing $\lambda\equiv-\beta\pmod{1}$ for 
the representative of $-\beta$ in $[0,1)$ (so that $\theta^{-1}=e^{2\pi i\alpha}$ is 
the multiplier attached to the shift $n\mapsto n+\lambda$ below, matching \cite{DLM}):
\begin{equation}
\label{merda}
P_1^{(1)}\left[
\begin{array}{c}
\theta \\
\phi
\end{array}
\right] (z,\tau )=- \sideset{}{'}\sum\limits_{n\in \mathbb{Z}
+\lambda}\frac{q_z^n}{1-\theta^{-1}q^n}=
\frac{1}{z}- \sum\limits_{n\geq 1}
E^{(1)}_n \left[ {
\begin{array}{c}
\theta \\
\phi
\end{array}
}\right] (\tau )z^{n-1},
\end{equation}
for $q=e^{2\pi i\tau }$, where $\sum^{\prime }$ means we omit $n=0$ if 
$(\theta, \phi)=(1,1)$, and:
\begin{align*}
E^{(1)}_n\left[ {
\begin{array}{c}
\theta \\
\phi
\end{array}
}\right] (\tau ) &=-\frac{B_n(\lambda )}{n!}+\frac{1}{(n-1)!}
\sideset{}{'}\sum\limits_{r\geq 0}\frac{(r+\lambda)^{n-1}\theta
^{-1}q^{r+\lambda}}{1-\theta^{-1}q^{r+\lambda }}  \notag \\
&\quad +\frac{(-1)^{n}}{(n-1)!}\sum\limits_{r\geq 1}\frac{(r-\lambda)^{n-1}
\theta q^{r-\lambda }}{1-\theta q^{r-\lambda}}.
\end{align*}

Additionally, for integers $m_i, n_j \geq 0$ satisfying 
$\sum_{i=1}^r m_i = \sum_{j=1}^s n_j$, let us introduce the notation:
\begin{align*}
& \Theta^{(1)}_{r, s,  (m_i, n_j) }(x, y, \tau)
\\
& =
\frac{ \prod\limits_{1\leq i\leq r,1\leq j\leq s} \vartheta^{(1)}\left[
\begin{array}{c}
\frac{1}{2} \\
\frac{1}{2}
\end{array}
\right] (x_i-y_j,\tau )^{m_i n_j}}
{ \prod\limits_{1\leq i<k\leq r}\vartheta^{(1)}\left[
\begin{array}{c}
\frac{1}{2} \\
\frac{1}{2}
\end{array}
\right](x_i-x_k, \tau)^{m_i m_k}
\prod\limits_{1\leq j<l\leq s}\vartheta^{(1)}\left[
\begin{array}{c}
\frac{1}{2} \\
\frac{1}{2}
\end{array}
\right](y_j-y_l,\tau )^{n_j n_l} }.
\end{align*}

Let us introduce ${\bf P}_n(\theta, \phi)$, an $n\times n$ matrix for 
$1\leq i, j\leq n$, with 
${\bf x}=\left(x_1, \ldots, x_n \right)$ and ${\bf y}=\left(y_1, \ldots, y_n \right)$:
\begin{equation}
\label{pmatrico}
S^{(1)}({\bf x}, {\bf y})={\bf P}_n (\theta, \phi)=\left[P^{(1)}_1 \left[
\begin{array}{c}
\theta \\
\phi
\end{array}
\right] (x_i-y_j,\tau )\right],
\end{equation}
and another $(n+1)\times (n+1)$ matrix ${\bf Q}_n(\tau)$:
\begin{equation}
{\bf Q}_n(\tau)=\left[
\begin{array}{cc}
{\bf P}_n(1,1) & \begin{matrix} 1 \\ \vdots \\ 1 \end{matrix}
\\
1 \ldots 1 & 0
\end{array}
\right]
=\left[
\begin{array}{cccc}
P^{(1)}_1(x_1-y_1, \tau) & \ldots & P^{(1)}_1(x_1-y_n,\tau ) & 1 \\
\vdots & \ddots &  & \vdots \\
P^{(1)}_1(x_n-y_1, \tau) &  & P^{(1)}_1(x_n-y_n,\tau ) & 1 \\
1 & \ldots & 1 & 0
\end{array}
\right].
\end{equation}

\begin{proposition}
\label{prop:genus1}
Let $n\geq1$, and let $x_1,\dots,x_n,y_1,\dots,y_n\in\mathbb C$ be generic in the sense 
that the $x_i$ are pairwise distinct modulo $\mathbb Z+\tau\mathbb Z$, likewise the $y_j$, 
and $x_i\not\equiv y_j\pmod{\mathbb Z+\tau\mathbb Z}$ for all $i,j$. Generalizing Garvan  
formula, for $(\theta, \phi) \ne (1,1)$, one has:
\begin{equation}
\label{GenGarvan_1}
\eta(\tau)^{3n}=
(-1)^{n(n-1)/2}\,i^{-n}\,
\frac{
\vartheta^{(1)} \left[
\begin{array}{c}
\frac{1}{2}-\beta \\
\frac{1}{2}+\alpha
\end{array}
\right] (0,\tau )
\;
\Theta^{(1)}_{n, n,  (1, 1)}(x, y,  \tau)
} {
\vartheta^{(1)} \left[
\begin{array}{c}
\frac{1}{2}-\beta \\
\frac{1}{2}+\alpha
\end{array}
\right] \left(\sum\limits_{i=1}^{n}(x_i-y_i), \tau \right)
} \;
\mathrm{\det } \; {\bf P}_{n} (\theta, \phi),
\end{equation}
equivalently, writing out $\det{\bf P}_n(\theta,\phi)$ via \eqref{merda},
\begin{multline}
\label{GenGarvan_1b}
\eta(\tau)^{3n} =
(-1)^{n(n-1)/2}\,i^{-n}\,
\frac{
\vartheta^{(1)} \left[
\begin{array}{c}
\frac{1}{2}-\beta \\
\frac{1}{2}+\alpha
\end{array}
\right] (0,\tau )
\;
\Theta^{(1)}_{n, n, (1, 1)}(x, y, \tau)
} {
\vartheta^{(1)} \left[
\begin{array}{c}
\frac{1}{2}-\beta \\
\frac{1}{2}+\alpha
\end{array}
\right] \left(\sum\limits_{i=1}^{n}(x_i-y_i), \tau \right)
} \; \times \\
\det
\left[
\frac{1}{x_i - y_j} - \sum\limits_{k\geq 1}
E^{(1)}_k \left[ {
\begin{array}{c}
\theta \\
\phi
\end{array}
}\right] (\tau ) (x_i - y_j)^{k-1}
\right]_{1 \le i, j \le n},
\end{multline}
and for $(\theta, \phi)=(1,1)$:
\begin{equation}
\label{GenGarvan_2}
\eta(\tau)^{3(n-1)}=
(-1)^{n(n-1)/2}\,i^{-(n+1)} \;
\frac{ \Theta^{(1)}_{n, n, (1, 1)}(x, y,  \tau) }{
\vartheta^{(1)} \left[
\begin{array}{c}
\frac{1}{2} \\
\frac{1}{2}
\end{array}
\right] \left(\sum\limits_{i=1}^{n}(x_i-y_i), \tau \right) } \;
\mathrm{\det } \;
{\bf Q}_{n}.
\end{equation}
\end{proposition}

Since $(-1)^{8m(8m-1)/2}=1$ and $i^{-8m}=1$ for every integer $m\geq1$, taking $n=8m$ 
in \eqref{GenGarvan_1}-\eqref{GenGarvan_1b} recovers a power of the discriminant 
$\Delta=\eta^{24}$; and since $(-1)^{(8m+1)8m/2}=1$ while $i^{-(8m+2)}=-1$, taking 
$n=8m+1$ in \eqref{GenGarvan_2} does the same:

\begin{corollary}
\label{cor:genus1}
For every $m\geq1$ and $(\theta, \phi) \ne (1,1)$,
\begin{equation}
\eta(\tau)^{24m}=
\frac{
\vartheta^{(1)} \left[
\begin{array}{c}
\frac{1}{2}-\beta \\
\frac{1}{2}+\alpha
\end{array}
\right] (0,\tau )
\;
\Theta^{(1)}_{8m, 8m,  (1, 1)}(x, y,  \tau)
} {
\vartheta^{(1)} \left[
\begin{array}{c}
\frac{1}{2}-\beta \\
\frac{1}{2}+\alpha
\end{array}
\right] \left(\sum\limits_{i=1}^{8m}(x_i-y_i), \tau \right)
} \;
\mathrm{\det } \; {\bf P}_{8m} (\theta, \phi),
\end{equation}
and
\begin{equation}
\eta(\tau)^{24m}=
-\frac{ \Theta^{(1)}_{8m+1, 8m+1, (1, 1)}(x, y,  \tau) }{
\vartheta^{(1)} \left[
\begin{array}{c}
\frac{1}{2} \\
\frac{1}{2}
\end{array}
\right] \left(\sum\limits_{i=1}^{8m+1}(x_i-y_i), \tau \right) } \;
\mathrm{\det } \;
{\bf Q}_{8m+1}.
\end{equation}
Since the left-hand sides do not depend on $x,y$, each of 
\eqref{GenGarvan_1}-\eqref{GenGarvan_2} and of this Corollary claims that a family of
 expressions depending on $16n$ (respectively $16m$) free complex parameters is in 
fact constant, equal to a fixed power of $\eta(\tau)$.
\end{corollary}

Formulas \eqref{GenGarvan_1}-\eqref{GenGarvan_2} can also be expressed purely in terms 
of deformed Eisenstein series by substituting the definition of 
$P^{(1)}_1\left[ \begin{smallmatrix} \theta \\ \phi \end{smallmatrix} \right](z, \tau)$ 
in terms of 
$E^{(1)}_n\left[ \begin{smallmatrix} \theta \\ \phi \end{smallmatrix} \right](\tau)$, 
leading to a quite involved formula which we do not give here.

Let us give a proof. Recall the genus one prime form \cite{Mumford, Fay}:
\[
K^{(1)}(x-y, \tau) =\frac{ \vartheta^{(1)} \left[ {{1/2} \atop {1/2}} \right]
\left( \int_y^x\nu, \tau \right) }
{\zeta(x)^\half \zeta(y)^\half},
\]
where 
$\zeta^{(1)}(x) = \sum_{i=1}^g \partial_{z_i}\vartheta^{(1)} 
\left[ { 1/2} \atop {1/2} \right] (0,  \Omega)\nu^{(1)}_i(x)$, 
 with the characteristic taken to be the (unique, up to sign) odd half-integer 
characteristic $\left[{1/2}\atop{1/2}\right]$ in genus one, which gives:
$K^{(1)}(z, \tau)
= - \frac{i}{\eta^3(\tau)}
\vartheta^{(1)}  \left[
\begin{array}{c}
\frac{1}{2} \\
\frac{1}{2}
\end{array}
\right] (z, \tau)$.  
In particular, for the parameters $(\alpha,\beta)$ of this section, at $\alpha=\beta=1/2$ 
the characteristic in \eqref{eq:Zform2} becomes 
$\left[{0\atop1}\right]=\left[{0\atop0}\right]$ and 
the phase $e^{2\pi i(\alpha+1/2)(\beta-1/2)}$ becomes $1$, so that
$$Z^{(1)} \left[{f_{1/2} \atop g_{1/2}} \right] \left( \tau \right) 
= \frac{\vartheta^{(1)}\left[{0\atop0}\right](0,\tau)}{\eta(\tau)},$$
an even theta constant over a single power of $\eta$. Being a partition function, 
the left side does not depend on an insertion point $z$, thus it cannot equal 
$K^{(1)}(z,\tau)/\eta^2(\tau)$, which does; the two also disagree in theta characteristic 
and in the power of $\eta$. The object that does carry the $z$-dependence of the prime form 
is the twisted genus one Szeg\H{o} kernel 
$S^{(1)}\left[\begin{smallmatrix}\theta\\\phi\end{smallmatrix}\right](x,y|\tau)
=P_1^{(1)}\left[\begin{smallmatrix}\theta\\\phi\end{smallmatrix}\right](x-y,\tau)$ of 
\eqref{merda}, which is what enters the genus two Szeg\H{o} kernel of Section 
\ref{sec:genus2} below.

The generating function for the torus $n$-point functions of the rank two fermion vertex 
operator superalgebra is $\det{\bf P}_n(\theta,\phi)$ 
\cite[Propositions \; 14, 15]{Mason-Tuite-Z}, and it is evaluated by Frobenius' determinant 
identity \cite{Frobenius}, in the elliptic form due to Fay \cite{Fay} 
(compare \cite[proof of Proposition\; 4.4]{TZ2}, where it is quoted in the form used here).
Introduce:
\[
K^{(1)}_n({\bf x}, {\bf y})=
\frac{\prod\limits_{1\leq i<j\leq n} K^{(1)}(x_i-x_j, \tau)
K^{(1)}(y_i-y_j, \tau)}{\prod\limits_{1\leq i,j\leq n}K^{(1)}(x_i-y_j,\tau)}.
\]
For $(\theta, \phi )\neq (1,1)$, one has:
\begin{equation}
\label{eq:frobP}
\mathrm{\det } \;
{\bf P}_n (\theta, \phi)
=(-1)^{n(n-1)/2}\,\frac{\vartheta^{(1)} \left[
\begin{array}{c}
\frac{1}{2}-\beta \\
\frac{1}{2}+\alpha
\end{array}
\right] \left(\sum\limits_{i=1}^n(x_i-y_i), \tau \right)}{\vartheta^{(1)} \left[
\begin{array}{c}
\frac{1}{2}-\beta \\
\frac{1}{2}+\alpha
\end{array}
\right] (0,\tau )}
K^{(1)}_n({\bf x}, {\bf y}),
\end{equation}
and similarly for $(\theta, \phi)=(1,1)$:
\begin{equation}
\label{eq:frobQ}
\det \;
{\bf Q}_n =
-(-1)^{n(n-1)/2}\, K^{(1)} \left(\sum\limits_{i=1}^n(x_i-y_i), \tau\right)
K^{(1)}_n({\bf x}, {\bf y}).
\end{equation}
The sign $(-1)^{n(n-1)/2}$ reflects the orientation of the factors $K^{(1)}(y_i-y_j)$ 
in $K^{(1)}_n$: it is the sign of the classical Cauchy determinant 
$\det\big[1/(x_i-y_j)\big]=\prod_{i<j}(x_i-x_j)(y_j-y_i)\big/\prod_{i,j}(x_i-y_j)$, 
which \eqref{eq:frobP} generalizes (replace $P_1^{(1)}$ by $1/z$, $K^{(1)}$ by $z$ and 
the theta quotient by $1$), $K^{(1)}$ being odd. It is $+1$ for every $n$ occurring in 
Corollary \ref{cor:genus1}; we have also checked \eqref{eq:frobP} and \eqref{eq:frobQ} 
numerically for $n\le5$. Since 
$K^{(1)}(z,\tau)=c\,\vartheta^{(1)}\left[\begin{smallmatrix}1/2\\1/2\end{smallmatrix}\right]
(z,\tau)$ with $c=-i/\eta^3(\tau)$, and from the definition of $\Theta^{(1)}_{n,n,(1,1)}$,
$$\frac{\prod_{i<j}\vartheta^{(1)}\left[\begin{smallmatrix}1/2\\1/2\end{smallmatrix}\right]
(x_i-x_j,\tau)\vartheta^{(1)}\left[\begin{smallmatrix}1/2\\1/2\end{smallmatrix}\right]
(y_i-y_j,\tau)}{\prod_{i,j}\vartheta^{(1)}\left[\begin{smallmatrix}1/2\\1/2\end{smallmatrix}
\right](x_i-y_j,\tau)} = \Theta^{(1)}_{n,n,(1,1)}(x,y,\tau)^{-1},$$
we get $K_n^{(1)}(\mathbf x,\mathbf y) = c^{-n}\Theta^{(1)}_{n,n,(1,1)}(x,y,\tau)^{-1} = 
i^n\eta(\tau)^{3n}\Theta^{(1)}_{n,n,(1,1)}(x,y,\tau)^{-1}$, since $c^{-1}=i\eta^3(\tau)$. 
Substituting into the two displayed identities and solving for a power of $\eta(\tau)$ gives 
\eqref{GenGarvan_1} and \eqref{GenGarvan_2}; \eqref{GenGarvan_1b} is \eqref{GenGarvan_1} with 
$\det{\bf P}_n(\theta,\phi)$ written out via \eqref{merda}.

There exists also the analytic expansion for $k, l\geq 1$:
\begin{align*}
P^{(1)}_1\left[ {
\begin{array}{c}
\theta \\
\phi
\end{array}
}\right] (z+z_1-z_2, \tau) &= \sum_{k, l\geq 1} D^{(1)}\left[ {
\begin{array}{c}
\theta \\
\phi
\end{array}
}\right] (k,l, z,\tau) z_1^{k-1}z_2^{l-1}, \\
D^{(1)}\left[ {
\begin{array}{c}
\theta \\
\phi
\end{array}
}\right] (k,l, z, \tau) &=(-1)^{k+1}
\binom{k+l-2}{k-1}
P^{(1)}_{k+l-1}\left[ {
\begin{array}{c}
\theta \\
\phi
\end{array}
}\right] (\tau, z).
\end{align*}
Introduce the block matrix:
\[
{\bf D}_{r,s}=\left[
\begin{array}{ccc}
{\bf  D}_{(11)} & \ldots & {\bf  D}_{(1s)} \\
\vdots & \ddots & \vdots \\
{\bf  D}_{(r1)} & \ldots & {\bf  D}_{(rs)}
\end{array}
\right],
\]
with ${\bf D}_{(ab)}$ the $m_a\times n_b$ matrix defined by
${\bf D}_{(ab)}(i,j)=\left[ D^{(1)}\left[ \begin{smallmatrix} \theta \\ \phi \end{smallmatrix} 
\right] (i,j, x_a-y_b, \tau)\right]$,
for $(1\leq i\leq m_a, 1\leq j\leq n_b)$, $1\leq a\leq r$, and $1\leq b\leq s$.

Using the full (multiplicity-weighted) version of Fay generalized trisecant identity 
\cite{Mason-Tuite-Z}, we derive the following:

\begin{proposition}
\label{prop:genus1multi}
Let $m_1,\dots,m_r,n_1,\dots,n_s\geq0$ be integers with 
$N:=\sum_{i=1}^r m_i=\sum_{j=1}^s n_j$, and let $x_1,\dots,x_r,y_1,\dots,y_s\in\mathbb C$ 
be generic as in Proposition \ref{prop:genus1}. For $(\theta, \phi) \ne (1,1)$:
\begin{equation}
\label{eq:genus1multi}
\eta(\tau)^{3\Phi}
=
(-1)^{N(N+1)/2}\,i^{\Phi}\;
\frac{
\vartheta^{(1)}
\left[
\begin{array}{c}
\frac{1}{2}-\beta \\
\frac{1}{2}+\alpha
\end{array}
\right] (0,\tau ) \; \Theta^{(1)}_{r, s, (m, n)} (x, y, \tau)}
{\vartheta^{(1)} \left[
\begin{array}{c}
\frac{1}{2}-\beta \\
\frac{1}{2}+\alpha
\end{array}
\right] \left(\sum\limits_{i=1}^r m_i x_i -\sum\limits_{j=1}^s n_j y_j, \tau \right)
}
\;
\det \; {\bf  D}_{r,s},
\end{equation}
where 
$\Phi= \sum_{1 \le i \le r,  1 \le j \le s} m_i n_j - \sum_{1 \le i  <k \le r} m_i m_k 
- \sum_{1 \le j < l\le s} n_j n_l$.
\end{proposition}

\begin{proof}
Write $N=\sum_im_i=\sum_jn_j$ and $e=\sum_i\binom{m_i}{2}+\sum_j\binom{n_j}{2}$; by the 
closed form of $\Phi$ proved at the end, $e=\Phi-N$. We first derive the confluent form 
of the Frobenius identity \eqref{eq:frobP}. Replace each $x_i$ by the $m_i$ points 
$x_i+\varepsilon t_{i,p}$, $1\le p\le m_i$, and each $y_j$ by the $n_j$ points 
$y_j+\varepsilon u_{j,q}$, $1\le q\le n_j$, with all the $t_{i,p}$ (respectively $u_{j,q}$) 
distinct, and let $\varepsilon\to0$ in the $N\times N$ identity \eqref{eq:frobP}. 
By the expansion $P^{(1)}_1\left[\begin{smallmatrix}\theta\\\phi\end{smallmatrix}\right]
(z+t-u,\tau)=\sum_{k,l\ge1}D^{(1)}\left[\begin{smallmatrix}\theta\\\phi\end{smallmatrix}
\right](k,l,z,\tau)\,t^{k-1}u^{l-1}$ above, the $N\times N$ matrix ${\bf P}_N$ is a product 
$A\,\mathcal D\,B$, with $A=\big[(\varepsilon t_{i,p})^{k-1}\big]$, 
$B=\big[(\varepsilon u_{j,q})^{l-1}\big]$ and 
$\mathcal D=\big[D^{(1)}(k,l,x_i-y_j,\tau)\big]$; by the Cauchy-Binet formula the leading 
term of $\det{\bf P}_N$ is $\varepsilon^{e}\,\Delta(t)\Delta(u)\det{\bf D}_{r,s}$, 
where $\Delta(t)=\prod_i\prod_{p<p'}(t_{i,p'}-t_{i,p})$ and 
$\Delta(u)=\prod_j\prod_{q<q'}(u_{j,q'}-u_{j,q})$. On the right-hand side of \eqref{eq:frobP}, 
$K^{(1)}(z,\tau)=z+O(z^3)$ shows that the factors of $K^{(1)}_N$ whose two arguments lie 
in the same cluster contribute 
$\varepsilon^{e}\prod(t_{i,p}-t_{i,p'})\prod(u_{j,q}-u_{j,q'})=(-1)^{e}\varepsilon^{e}
\Delta(t)\Delta(u)$, the remaining factors tend to those of
$$K^{(m,n)}_{r,s}(\mathbf x,\mathbf y) := 
\frac{\prod_{1\le i<k\le r}K^{(1)}(x_i-x_k,\tau)^{m_im_k}\prod_{1\le j<l\le s}K^{(1)}
(y_j-y_l,\tau)^{n_jn_l}}{\prod_{1\le i\le r,1\le j\le s}K^{(1)}(x_i-y_j,\tau)^{m_in_j}},$$
and the theta quotient tends to $\vartheta^{(1)}\left[\begin{smallmatrix}1/2-\beta\\1/2
+\alpha\end{smallmatrix}\right]\!\big(\sum m_ix_i-\sum n_jy_j,\tau\big)/\vartheta^{(1)}
\left[\begin{smallmatrix}1/2-\beta\\1/2+\alpha\end{smallmatrix}\right](0,\tau)$. 
Comparing leading terms,
$$\det{\bf D}_{r,s}=(-1)^{N(N-1)/2+e}\,\frac{\vartheta^{(1)}
\left[\begin{smallmatrix}1/2-\beta\\1/2+\alpha\end{smallmatrix}\right]
\!\big(\sum m_ix_i-\sum n_jy_j,\tau\big)}{\vartheta^{(1)}\left[\begin{smallmatrix}1/2
-\beta\\1/2+\alpha\end{smallmatrix}\right](0,\tau)}\,K^{(m,n)}_{r,s}(\mathbf x,\mathbf y),$$
and $(-1)^{N(N-1)/2+e}=(-1)^{N(N+1)/2+\Phi}$ since $e=\Phi-N$. As in the proof of 
Proposition \ref{prop:genus1}, substituting 
$K^{(1)}=c\,\vartheta^{(1)}\left[\begin{smallmatrix}1/2\\1/2\end{smallmatrix}\right]$, 
$c=-i/\eta^3(\tau)$, and using the definition of $\Theta^{(1)}_{r,s,(m,n)}$ gives
$$K^{(m,n)}_{r,s}(\mathbf x,\mathbf y) = c^{-\Phi}\,\Theta^{(1)}_{r,s,(m,n)}
(\mathbf x,\mathbf y,\tau)^{-1} = i^{\Phi}\eta(\tau)^{3\Phi}\,\Theta^{(1)}_{r,s,(m,n)}
(\mathbf x,\mathbf y,\tau)^{-1}.$$
Solving for $\eta(\tau)^{3\Phi}$ and using $(-1)^{\Phi}i^{-\Phi}=i^{\Phi}$ gives 
\eqref{eq:genus1multi}. The closed form $\Phi=\tfrac12\big(\sum_im_i^2+\sum_jn_j^2\big)$ 
follows from $N^2=\big(\sum_im_i\big)\big(\sum_jn_j\big)=\sum_{i,j}m_in_j$ together with 
$\sum_{i<k}m_im_k=\tfrac12(N^2-\sum_im_i^2)$ and $\sum_{j<l}n_jn_l=\tfrac12(N^2-\sum_jn_j^2)$;
 in particular $e=\tfrac12\big(\sum_im_i^2+\sum_jn_j^2\big)-N=\Phi-N$.
\end{proof}

\begin{remark}
Taking $r=s=n$ and all $m_i=n_j=1$ (thus $N=\Phi=n$) recovers \eqref{GenGarvan_1}, 
since $(-1)^{n(n+1)/2}i^n=(-1)^{n(n-1)/2}i^{-n}$ identically. Unlike Proposition 
\ref{prop:genus1}, the sign in \eqref{eq:genus1multi} is not determined by 
$\Phi$ alone: at $\Phi=8$, for instance, $\mathbf m=\mathbf n=(2,2)$ gives $N=4$ 
and prefactor $(-1)^{10}i^{8}=+1$, while $\mathbf m=(2,2,1)$, $\mathbf n=(2,1,1,1)$ 
gives $N=5$ and prefactor $(-1)^{15}i^{8}=-1$; in both cases the left side of 
\eqref{eq:genus1multi} is $\eta(\tau)^{24}=\Delta(\tau)$, and in general, for $\Phi=8m$, 
it is $\Delta(\tau)^m$ with a prefactor $\pm1$ that depends on $N$. Finally, as in
 Corollary \ref{cor:genus1}, the left side of \eqref{eq:genus1multi} does not depend on 
$x,y$: it claims that a family of expressions depending on $16N$ free complex parameters 
is constant.
\end{remark}

%%%%%%%%%%%%%%%%%%%%%%%%%%%%%%%%%%%%%%%%%%%%%%%%%%%%%%%%%%%%%%%%%%%%%%%%%%%%
\section{Genus Two Formulas}\label{sec:genus2}

In this section we use the parametrization of \cite{TZ1,TZ2}, in which the genus one 
partition function of the rank two fermion reads \cite[Section 5.4]{TZ1}
\begin{equation}
\label{eq:Zdirect}
Z^{(1)}\left[{f_a \atop g_a} \right](\tau_a)=e^{-2\pi i\alpha_a\beta_a}\,
\frac{\vartheta^{(1)}\left[ {\alpha_a \atop \beta_a}\right](0,\tau_a)}{\eta(\tau_a)},
\qquad \theta_a=-e^{-2\pi i\beta_a},\quad \phi_a=-e^{2\pi i\alpha_a}.
\end{equation}
This is \eqref{eq:Zform2} with the parameters of Section \ref{sec:garvan} renamed, 
$(\alpha,\beta)_{\text{here}}
=(\tfrac12-\beta,\tfrac12+\alpha)_{(Section \; \ref{sec:garvan})}$:
 the characteristic is the same, $e^{-2\pi i\alpha\beta}
=e^{2\pi i(\alpha_{Section \; \ref{sec:garvan}}+1/2)(\beta_{Section \; 
\ref{sec:garvan}}-1/2)}$, and the multipliers coincide, 
$-e^{-2\pi i\beta}=e^{-2\pi i\alpha_{Section \; \ref{sec:garvan}}}$, 
$-e^{2\pi i\alpha}=e^{-2\pi i\beta_{Section \; \ref{sec:garvan}}}$. In particular 
$(\theta_a,\phi_a)=(1,1)$ corresponds here to $(\alpha_a,\beta_a)\equiv(1/2,1/2)$. 
In \cite{TZ1}, we derived the genus two counterpart of the triple Jacobi identity by 
comparing the partition function of the rank two fermion on a genus two Riemann surface, 
for $(\theta_a,\phi_a)\neq(1,1)$, $a=1,2$ \cite[Theorem\; 5.1]{TZ1}:
\begin{equation*}
Z^{(2)}\left[{f \atop g} \right](\tau_1, \tau_2, \epsilon)=
Z^{(1)}\left[{f_1 \atop g_1} \right](\tau_1) \;
Z^{(1)}\left[{f_2 \atop g_2} \right] (\tau_2) \; \det  \left(I - Q^{(1)} \right)
\end{equation*}
(for the rank one fermion the exponent of the determinant is $1/2$ instead, 
\cite[Corollary\; 5.4]{TZ1}), with its bosonized version \cite[(61)]{TZ1}, which uses the 
genus two Heisenberg partition function of \cite{MasonTuiteGenus2-I}:
\[
Z^{(2)}_M\left[{f \atop g} \right](\tau_1, \tau_2, \epsilon)
=e^{-2\pi i\alpha\cdot\beta}\,\frac{\vartheta^{(2)}
\left[ {\alpha \atop \beta}\right] \left(\Omega^{(2)}\right)}
{\eta(\tau_1) \eta(\tau_2) \det \left(I-A_1A_2 \right)^{1/2}},
\]
with column vectors $\alpha=(\alpha_1, \alpha_2)^t$ and $\beta=(\beta_1, \beta_2)^t$. 
Here for $a=1, 2$:
\[
Q^{(1)}=\left(
\begin{array}{cc}
0 & \xi F^{(1)}_1 \left[ {\theta_1}  \atop {\phi_1} \right] \\
- \xi F^{(1)}_2 \left[ {\theta_2}  \atop {\phi_2} \right] & 0
\end{array}
\right),
\]
\[
F^{(1)}_a \left[ {\theta_a}  \atop {\phi_a} \right](k,l,\tau_a, \epsilon)
=(-1)^l \epsilon^{\frac{1}{2}(k+l-1)}
\binom{k+l-2}{k-1}E^{(1)}_{k+l-1}\left[{\theta_a \atop \phi_a}\right] (\tau_a),
\]
\begin{equation*}
A^{(1)}_a(k,l, \tau_a, \epsilon)
=\epsilon^{(k+l)/2}\frac{(-1)^{k+1}(k+l-1)!}{\sqrt{kl}(k-1)!(l-1)!}
E^{(1)}_{k+l}(\tau_a),
\end{equation*}
Equating the two expressions for the genus two partition function, the phases 
$e^{-2\pi i\alpha_a\beta_a}$ and the factor $\eta(\tau_1)\eta(\tau_2)$ cancel 
identically against those of \eqref{eq:Zdirect}, which gives the genus two Jacobi 
product identity \cite[Theorem\; 5.11]{TZ1}
\begin{equation}
\label{eq:genus1to2theta}
\vartheta^{(2)}\left[{\alpha \atop \beta}\right] \left(\Omega^{(2)}\right)
=
\vartheta^{(1)}\left[{\alpha_1 \atop \beta_1}\right](0,\tau_1)\,
\vartheta^{(1)}\left[{\alpha_2 \atop \beta_2}\right](0,\tau_2)\,
\det\left(I-Q^{(1)}\right)\det\left(I-A_1A_2\right)^{1/2},
\end{equation}
valid for $(\theta_a,\phi_a)\neq(1,1)$, $a=1,2$. At $\tau_1=\tau_2=\tau$ no power 
of $\eta(\tau)$ appears in \eqref{eq:genus1to2theta}, and the prime form $K^{(1)}$ 
does not enter: this comparison does not isolate a power of $\eta$. 
(The degenerate case $(\theta_2,\phi_2)=(1,1)$, i.e., $\alpha_2=\beta_2=1/2$, 
where $Z^{(1)}[f_2,g_2]=0$ because 
$\vartheta^{(1)}\left[\begin{smallmatrix}1/2\\1/2\end{smallmatrix}\right](0,\tau_2)=0$, 
is excluded in \cite{TZ1}, and we do not use it.) The genus two determinant formula for 
a power of $\eta(\tau)$ that this note establishes is Proposition \ref{prop:genus2} below, 
obtained independently via the genus two Jacobi triple identity of \cite{TZ2}.

Here we construct the genus two analogue by torus self-sewing \cite{TZ1,TZ2}. Two different 
sewing schemes appear in what follows: $\epsilon$ (in $Z^{(2)}(\tau_1,\tau_2,\epsilon)$, 
$F^{(1)}_a$, $A^{(1)}_a$ above) parametrizes the genus two surface obtained by sewing 
 two tori together, while $\rho$ below parametrizes the genus two surface obtained by 
 self-sewing a single torus (attaching a handle); both constructions realize the same 
genus two Riemann surface, in the sense of \cite{TZ1,TZ2}, but the two sewing parameters 
are not identified with one another in what follows. Likewise, the parameter 
$\kappa$ introduced below is the twist parameter for the self-sewn handle and plays 
a role formally analogous to, but independent of, the $\kappa=\beta-1/2$ of Section 
\ref{sec:garvan}; and the characteristics $\alpha=(\alpha_1,\alpha_2)^t$, 
$\beta=(\beta_1,\beta_2)^t$ of this section are genus two theta characteristics 
in the direct (unshifted) convention of \cite{TZ1,TZ2}, unlike the shifted parametrization 
$a=1/2-\beta$, $b=1/2+\alpha$ used for genus one characteristics in Section \ref{sec:garvan} 
(see the dictionary above).

In \cite[Theorem\; 5.9]{TZ2}, obtained by comparing the fermionic genus two partition function 
\cite[Theorem\; 5.1]{TZ2} with its bosonized form (which uses the Heisenberg genus two 
partition function of \cite[Theorems\; 5.1 and 5.6]{MasonTuiteGenus2-II}), we found a genus 
two analogue of the classical Jacobi triple product identity:
\begin{equation}
\label{eq:genus2JTP}
\frac{
\vartheta^{(2)}
\begin{bmatrix}
\alpha \\ \beta
\end{bmatrix}
\left( \Omega^{(2)}\right)
}
{
\vartheta^{(1)}
\begin{bmatrix}
\alpha_1 \\ \beta_1
\end{bmatrix}
\left(\kappa w, \tau \right)
}
=
e^{2 i\pi \beta_2\kappa}
\left(\frac{e^{i\pi B }\rho}{ K^{(1)} (w, \tau)^2 }\right)^{\half \kappa^2}
\det \left(I - T^{(2)} \right) \det \left(I - R \right)^\half.
\end{equation}
Here $\rho$ is the torus self-sewing complex parameter, with $(\tau,w,\rho)$ in the sewing 
domain 
$\mathcal D^\rho=\{(\tau,w,\rho):\tau\in\mathbb H,
\ |w-\ell|>2|\rho|^{1/2}>0\ \text{for all}\ \ell\in\Lambda
=2\pi i(\mathbb Z\tau\oplus\mathbb Z)\}$, $-1/2 < \kappa <1/2$ and $(\theta_1,\phi_1)
\neq(1,1)$; $B$ is an odd integer (not to be confused with the Bernoulli polynomials 
$B_n(\lambda)$ of Section \ref{sec:garvan}), fixing the branch of each of the non-integral 
powers $\eta^{3\kappa^2}$, $\vartheta^{(1)}[\cdot](w,\tau)^{\kappa^2}$ and 
$(-e^{i\pi B}\rho)^{\kappa^2/2}$ occurring below; $T^{(2)}=\xi G^{(2)}D(\theta_2)$; 
and $\xi\in\{\pm \sqrt{-1}\}$.
The matrix $R$ is defined by:
\[
R_{ab}(k,l)=-\frac{\rho^{(k+l)/2}}{\sqrt{kl}}
\begin{bmatrix}
\widetilde D^{(1)}(k,l,w, \tau) & C^{(1)}(k,l, \tau) \\
C^{(1)}(k,l,\tau ) & \widetilde D^{(1)}(l,k, w, \tau)
\end{bmatrix},
\]
with
\[
D(\theta_2)(k,l)=
\left[
\begin{array}{cc}
\theta_2^{-1} & 0 \\
0 & -\theta_2
\end{array}
\right]
\delta(k,l),
\]
\begin{align*}
C^{(1)}(k,l,\tau) &= (-1)^{k+1}\frac{(k+l-1)!}{(k-1)!(l-1)!} E^{(1)}_{k+l}(\tau), \\
\widetilde D^{(1)}(k,l, z, \tau) &= (-1)^{k+1}\frac{(k+l-1)!}{(k-1)!(l-1)!} 
P^{(1)}_{k+l}(\tau, z),
\end{align*}
the latter a genus-two-sewing analogue of, and not to be confused with, the matrix 
$D^{(1)}\left[\begin{smallmatrix}\theta\\\phi\end{smallmatrix}\right](k,l,z,\tau)
=(-1)^{k+1}\binom{k+l-2}{k-1}P_{k+l-1}^{(1)}(z,\tau)$ used in $\mathbf D_{r,s}$ of 
Section \ref{sec:garvan} above.
Writing $\overline a$ for the complementary index to $a\in\{1,2\}$ (thus $\overline1=2$, 
$\overline2=1$), putting $k_a=k+\kappa(-1)^{\overline a}$ and 
$l_b=l+\kappa(-1)^{\overline b}$ for integers $k,l\ge1$, and letting $\mathcal C_a(z_a)$ 
denote a closed anticlockwise contour, parametrized by $z_a$, surrounding the puncture 
$z_a=0$ in the annulus of the sewing construction \cite[Section \; 4]{Tuite-Z}, 
the matrix is:
\[
G^{(2)}_{ab}=\left[
\frac{ \rho^{\half (k_a+l_b-1)} } { (2\pi i)^2 }
\oint_{\mathcal{C}_{{\overline{a}}}(x_{\overline{a}})}
\oint_{\mathcal{C}_b(y_b)}
(x_{\overline{a}})^{-k_a} (y_b)^{-l_b}
S^{(2)}_\kappa (x_{\overline{a}}, y_b)
\;
dx_{\overline{a}}^{\half } \; dy_b^\half
\right].
\]
The genus two Szeg\H{o} kernel for $x, y$ taken on the torus is given by 
\cite[Theorem\; 4.6 and Section \; 4.4]{Tuite-Z} (restated as \cite[Theorem\; 2.1]{TZ2}):
\begin{align}
\label{eq:Szego_Genus2}
S^{(2)}(x,y) &=S^{(2)}_{\kappa}(x,y) +\xi h(x)D(\theta_2)
\left(I-T^{(2)}\right)^{-1} \overline{h}^t(y) \nonumber \\
&= \left((x-y)^{-1}
+ \sum\limits_{k,l\geq1} E^{(2)}_{k,l}\left(\Omega^{(2)} \right)x^{k-1}y^{l-1}
\right) dx^\half dy^\half,
\end{align}
where the second line is the expansion of the regular part in local coordinates $x,y$ 
centred at a fixed point of the torus away from the punctures, which defines the 
coefficients $E^{(2)}_{k,l}(\Omega^{(2)})$ (the genus two analogues of the coefficients 
$C^{(1)}(k,l)$ of the genus one expansion of $P^{(1)}_1$), and where $h(x)$, 
$\overline h(y)$ are the row vectors, indexed by $(k,a)$, with entries
\[
h_a(x,k)=\frac{\rho^{\half (k_a -\half )}}{2\pi i}
\oint_{\mathcal{C}_a(y_a)}
y_a^{-k_a}
S^{(2)}_\kappa (x, y_a) dy_a^\half,
\]
\[
\overline{h}_a(y,k)=\frac{\rho^{\half (k_a -\half )}}{2\pi i}
\oint_{\mathcal{C}_{\overline a}(x_{\overline a})}
x_{\overline a}^{-k_a}
S^{(2)}_\kappa (x_{\overline a}, y) dx_{\overline a}^\half,
\]
and $\overline{h}^t(y)$ denotes the transpose of $\overline h(y)$. We also use the 
semi-infinite matrices $H=\left( \left(h(x_i)\right) (k,a) \right)$, 
$\overline{H}^t=\left(\left( \overline{h} (y_i) \right) (l,b) \right)^t$. The rows of 
$H$ are indexed by $i$ and columns by $k\ge 1$ and $a=1, 2$. $\overline{H}^t$ is 
semi-infinite with rows indexed by $l\ge 1$ and $b=1, 2$, and with $n$ columns 
indexed by $j$.

We also introduce the matrices, with $a=1$ fixed:
\[
S^{(2)}_\kappa(x,y)=
\left(
\frac{
\vartheta^{(1)}
\begin{bmatrix}
\half
\\
\half
\end{bmatrix} (x-w,\tau) \vartheta^{(1)}
\begin{bmatrix}
\half
\\
\half
\end{bmatrix} (y,\tau)}
{\vartheta^{(1)}
\begin{bmatrix}
\half
\\
\half
\end{bmatrix} (x,\tau)  \vartheta^{(1)}
\begin{bmatrix}
\half
\\
\half
\end{bmatrix} (y-w,\tau)}
\right)^{\kappa}\,
\frac{ \vartheta^{(1)}  \begin{bmatrix}{\alpha_1} \\ {\beta_1} \end{bmatrix}
\left( x-y +\kappa w,\tau\right) dx^\half dy^\half}
{\vartheta^{(1)} \begin{bmatrix}{\alpha_1 } \\ {\beta_1}\end{bmatrix}
\left(\kappa w,\tau \right)  K^{(1)}(x-y,\tau)},
\]
\[
S^{(2)}_n({\bf x}, {\bf y})=\left[S^{(2)}(x_i, y_j)\right],
\quad
S_{\kappa, n}^{(2)}({\bf x}, {\bf y}) =\left[S^{(2)}_{\kappa}(x_i,y_j)\right],
\]
which are finite matrices with $1 \le i, j \le n$, and
\begin{equation*}
T=T^{(2)}=\xi G^{(2)}D(\theta_2).
\end{equation*}

In \cite[Proposition\; 5.4]{TZ2} (the $\rho$-formalism analogue of 
\cite[Proposition 5.5]{TZ1}), we also proved the following formula:
\begin{equation}
\label{eq:genus2detfact}
\det \left[
\begin{array}{cc}
S_{\kappa, n}^{(2)} & -\xi  HD(\theta_2)
\\
\overline{H}^t & I - T^{(2)}
\end{array}
\right] =
\det \left[ S^{(2)}_n
\begin{bmatrix}
\alpha \\ \beta
\end{bmatrix}
({\bf x}, {\bf y}) \right]
\det \left(I-T^{(2)} \right).
\end{equation}

With this notation in place:

\begin{proposition}
\label{prop:genus2}
For $n\ge 1$, $(\tau,w,\rho)\in\mathcal D^\rho$ and generic 
$x_1,\dots,x_n,y_1,\dots,y_n$ on the torus (so that $\det S^{(2)}_n\neq0$), 
with $(\theta_1,\phi_1)\neq(1,1)$ and 
$\vartheta^{(1)}\left[\begin{smallmatrix}\alpha_1\\\beta_1\end{smallmatrix}\right]
(\kappa w,\tau)\neq0$, one has the genus two determinant identity
\begin{align}
\eta^{3\kappa^2}(\tau)
&=
\frac{
e^{-2 i\pi \beta_2 \kappa}
\vartheta^{(2)}
\begin{bmatrix}
\alpha \\ \beta
\end{bmatrix}
\left( \Omega^{(2)} \right)
\vartheta^{(1)}
\left[
\begin{array}{c}
\frac{1}{2} \\
\frac{1}{2}
\end{array}
\right] (w, \tau)^{\kappa^2}
\det \left[ S^{(2)}_n\begin{bmatrix}
\alpha \\ \beta
\end{bmatrix} ({\bf x}, {\bf y}) \right]
}
{ \left( -e^{i\pi B } \rho\right)^{\half \kappa^2}  \vartheta^{(1)}
\begin{bmatrix}
\alpha_1 \\ \beta_1
\end{bmatrix}
\left(\kappa w, \tau \right)
\det \left(I - R\right)^{\half}
\det \left[
\begin{array}{cc}
S_{\kappa,n}^{(2)} & -\xi  HD(\theta_2)
\\
\overline{H}^t & I - T^{(2)}
\end{array}
\right]
}.
\end{align}
\end{proposition}

\begin{proof}
Solving \eqref{eq:genus2JTP} for $K^{(1)}(w,\tau)^{\kappa^2}$, using 
$K^{(1)}(w,\tau)
=-\frac{i}{\eta^3(\tau)}\vartheta^{(1)}\left[\begin{smallmatrix}1/2\\1/2
\end{smallmatrix}\right](w,\tau)$ and $i^{\kappa^2}=(-1)^{\kappa^2/2}$ - 
which converts $(e^{i\pi B}\rho)^{\kappa^2/2}$ into $(-e^{i\pi B}\rho)^{\kappa^2/2}$ 
- gives $\eta^{3\kappa^2}(\tau)$ as the displayed expression with $\det(I-T^{(2)})$ 
in place of the block determinant in the denominator; substituting the factorization 
\eqref{eq:genus2detfact} for $\det(I-T^{(2)})$ gives the stated formula.
\end{proof}

\begin{remark}
By the restriction $-1/2<\kappa<1/2$, one has $3\kappa^2\in[0,3/4)$: 
Proposition \ref{prop:genus2} never gives $\Delta=\eta^{24}$, nor any positive 
integral power of $\eta$, for any value of $\kappa$ in its stated range. It is, 
as written, a genus two determinant identity for $\eta(\tau)^{3\kappa^2}$ - a genus 
 one quantity expressed via genus two sewing data - instead of a genus 
two counterpart of the modular discriminant. Identifying the genus two modular form 
(conceivably related to the Igusa cusp form $\Delta_{10}(\Omega^{(2)})$ of weight $10$ 
\cite{Igusa1,Igusa2}) to which \eqref{eq:genus2JTP}-\eqref{eq:genus2detfact} degenerate,
if any, is left for future work.
\end{remark}

%%%%%%%%%%%%%%%%%%%%%%%%%%%%%%%%%%%%%%%%%%%%%%%%%%%%%%%%%%%%%%%%%%%%%%
\section{Discussion}

In this note we have given a vertex-operator-algebraic route to Garvan-type determinant 
formulas for powers of $\eta(\tau)$: Proposition \ref{prop:genus1} 
(and Corollary \ref{cor:genus1}) at genus one, Proposition \ref{prop:genus1multi} 
for its confluent/multiplicity generalization, and Proposition \ref{prop:genus2} 
at genus two. As emphasized in Section \ref{sec:classical-garvan}, the entries of 
our determinants are not modular forms, and \eqref{eq:Garvan_3x3} is not literally recovered; 
what is recovered, at genus one, is the classical Garvan-Milne formula for $\Delta^m$, 
now obtained from a prime-form count on a twisted Szeg\H{o} kernel instead of from 
the ring structure of $M_*(\mathrm{SL}_2(\mathbb Z))$.

At genus two, the matrix 
$S_n^{(2)}\left[\begin{smallmatrix}\alpha\\\beta\end{smallmatrix}\right]$ appearing 
in Proposition \ref{prop:genus2} has entries that are values of the genus two 
Szeg\H{o} kernel \eqref{eq:Szego_Genus2}; the coefficients $E^{(2)}_{k,l}(\Omega^{(2)})$ 
appear only after its expansion there, and it is in this heuristic sense only - not as a 
proved statement - that $S_n^{(2)}$ may be regarded as a genus two analog of 
$\begin{pmatrix}E_4&E_6\\E_6&E_8\end{pmatrix}$. As mentioned in the Remark following 
Proposition \ref{prop:genus2}, the Proposition expresses $\eta(\tau)^{3\kappa^2}$, 
not the modular discriminant, as a determinant built from these deformed Eisenstein series; 
whether \eqref{eq:genus2JTP}-\eqref{eq:genus2detfact} degenerate to an identity for a genus 
two modular form - the Igusa cusp form $\Delta_{10}(\Omega^{(2)})$ \cite{Igusa1,Igusa2} 
being the candidate - is a question we leave open, together with the companion question, 
raised in Section \ref{sec:classical-garvan}, of whether a finite Hankel minor of 
$\det(I-A_1A_2)$ gives a direct proof of Milne theorem.

This higher-genus formalism connects to related developments in the cohomological study 
of foliations on Riemann surfaces \cite{Zu3, Zu1, Zu}, and, for the further extension to 
arbitrary genus $g$, to the Zhu recursion formalism of \cite{WT}.

%%%%%%%%%%%%%%%%%%%%%%%%%%%%%%%%%%%%%%%%%%%%%%%%%%%%%%%%%%%%%%%%%%%%%%
\section*{Acknowledgments}

 We thank the referee for a careful reading
that corrected a number of errors in the original manuscript and, in particular, for
pointing out the connection between the Hankel matrices $C^{(1)}(k,l,\tau)$,
$A_a^{(1)}(k,l,\tau_a,\epsilon)$ of Section \ref{sec:genus2} and Milne Eisenstein series 
Hankel
determinants, mentioned in Section \ref{sec:classical-garvan} above. 

The last author is supported by the Institute of Mathematics of the Academy of 
Sciences of the Czech Republic (RVO 67985840), and would like to thank 
G. Mason and M. Tuite for useful discussions.

%% 
%%%%%%%%%%%%%%%%%%%%%%%%%%%%%%%%%%%%%%%%%%%%%%%%%%%%%%%%%%%%%%%%%%%%%%
%% 

\end{document}